\documentclass[12pt]{amsart}

\usepackage{graphicx}
\usepackage{amsfonts}
\usepackage{epsf}
\usepackage{amssymb}
\usepackage{amsmath}
\usepackage{amscd}

\newtheorem{theorem}{Theorem}[section]

\newtheorem{lemma}[theorem]{Lemma}

\newtheorem{corollary}[theorem]{Corollary}

\theoremstyle{remark} 

\newtheorem{example}[theorem]{Example}
\newtheorem{remark}[theorem]{Remark}

\newcommand{\s}{\mathfrak{s}}

\input xy
\xyoption{all}

\begin{document}

\title{Heegaard Floer correction terms and Dedekind-Rademacher sums}
\begin{abstract}
We derive a closed formula for the Heegaard Floer correction terms of lens spaces in terms of the classical Dedekind sum and its generalization, the Dedekind-Rademacher sum. Our proof relies on a reciprocity formula for the correction terms established by Ozsv\'ath and Szab\'o. A consequence of our result is that the Casson-Walker invariant of a lens space equals the average of its Heegaard-Floer correction terms. Additionally, we find an obstruction for the equality and equality with opposite sign, of two correction terms of the same lens space. Using this obstruction we are able to derive an optimal upper bound on the number of vanishing correction terms of lens spaces with square order second cohomology. 
\end{abstract}
\author{Stanislav Jabuka}
\email{jabuka@unr.edu}
\address{Department of Mathematics and Statistics, University of Nevada, Reno NV 89557.}
\thanks{S. Jabuka was partially supported by the NSF grant DMS 0709625.}

\author{Sinai Robins}
\email{rsinai@ntu.edu.sg}
\address{Division of Mathematical Sciences, School of Physical and Mathematical Sciences, Nanyang Technological University, Singapore, 637371.}

\author{Wang Xinli}
\email{wang0480@e.ntu.edu.sg}
\address{Division of Mathematical Sciences, School of Physical and Mathematical Sciences, Nanyang Technological University, Singapore, 637371.}

\maketitle
\section{Introduction and statement of results}
In \cite{OzsvathSzabo4} P. Ozsv\'ath and Z. Szab\'o introduced the {\em Heegaard Floer correction terms $d(Y,\s)\in \mathbb Q$} associated to a rational homology $3$-sphere $Y$ and a choice of a spin$^{\mathbb C}$-structure $\s$ on $Y$. Since their inception, the correction terms have proven to be a powerful tool in low-dimensional topology and have found applications to concordance questions \cite{OwensManolsecu, JabukaNaik, GrigsbyRubermanStrle},  $4$-manifold questions \cite{OzsvathSzabo4, OwensStrle} and others. 

When the $3$-manifold $Y$ is the lens space $L(p,q)$, it is customary to identify the set of spin$^{\mathbb C}$-structures on $L(p,q)$ with $H^2(L(p,q);\mathbb Z) \cong \mathbb Z/ p \mathbb Z $, allowing us to write $d(L(p,q),\s)$ as $d(L(p,q),n)$, $n\in \mathbb Z/p\mathbb Z$.  In \cite{OzsvathSzabo4}, Ozsv\'ath and Szab\'o proved a {\em reciprocity formula} for $d(L(p,q),n)$, valid under the restrictions $0<q<p$ and $0\le n <p+q$: 
\begin{equation} \label{ReciprocityForCorrectionTerms}
d(L(p,q),n) + d(L(q,p),n) =  \frac{1}{4} +\frac{n^2}{pq} +n\left( \frac{1}{pq} -\frac{1}{q}-\frac{1}{p} \right)+\frac{1}{4}\left(\frac{p}{q}+\frac{1}{pq} +\frac{q}{p}\right)-\frac{1}{2q}-\frac{1}{2p}.
\end{equation}
This formula is sufficient for a complete determination of $d(L(p,q),n)$, after taking into account that $L(p,q+kp)\cong L(p,q)$ for any $k\in \mathbb Z$, and that $d(L(1,1),0) = 0$. 

Reciprocity formulas such as \eqref{ReciprocityForCorrectionTerms} are common in the theory of Dedekind and Dedekind-Rademacher sums, indeed they are a crucial tool for the study of these number-theoretic functions. Before proceeding, we remind the reader of the definitions of the latter.   Let $\mathbb R \ni x\mapsto \bar B_1(x)$ be the function given by the {\em $1$-periodic first Bernoulli polynomial $\bar B_1(x) = \{x\} -\frac{1}{2}$}, where $\{x\} = x-\lfloor x \rfloor$ is the decimal part of $x\in \mathbb R$ (with $x\mapsto \lfloor x \rfloor$ as usual being the {\em floor function}). For relatively prime and nonzero integers $p,q$ and for any integer $n$, we define the {\em Dedekind-Rademacher sum $ s(q,p;n)$} as 
$$s(q,p;n) = \sum _{k=0}^{|p|-1} \bar B_1\left(\frac{kq+n}{p}\right) \bar B_1\left(\frac{k}{p}\right).$$ 
We note that $s(q,p;n)$ only depends on $q$ and $n$ through their moduli with respect to $p$. Recall that the {\em sawtooth function $\mathbb R\ni x\mapsto ((x))$} is defined as 
$$ ((x)) = \left\{
\begin{array}{cl}
x-\lfloor x \rfloor - \frac{1}{2} & \quad ; \quad x\notin \mathbb Z, \cr
0 & \quad ; \quad x\in \mathbb Z,
\end{array}
\right.
$$
giving rise to the classical {\em Dedekind sum $s(q,p)$}, defined for relatively prime nonzero integers $q,p$ as 
$$s(q,p) = \sum _{k=0}^{p-1} \left(\left( \frac{kq}{p}\right)\right) \cdot \left(\left( \frac{k}{p}\right)\right).$$
The Dedekind sum also only depends on $q$ through its modulus with respect to $p$. We invite the reader to check the easy relation $s(q,p;0)-s(q,p) = \frac{1}{4}$. 
\begin{remark} \label{RemarkOnDifferentDedekindRademacherSum}
There is a slightly different version of the Dedekind-Rademacher sum that appears in the literature, see for instance \cite{BeckRobins2, BeckRobins, Knuth}. Namely one can set 
$$\sigma (q,p;n) = \sum _{k=0}^{|p|-1} \left(\left( \frac{kq+n}{p}\right)\right) \cdot \left(\left( \frac{k}{p}\right)\right).$$
It is not hard to show that $s(q,p;n)$ and $\sigma(q,p;n)$ are related by 
$$s(q,p;n) - \sigma(q,p;n) = \left\{
\begin{array}{cl}
-\frac{1}{2}\left( \bar B_1 \left(\frac{n}{p}\right)+\bar B_1\left( \frac{k_q}{p}\right) \right) & \quad ;  \quad  n\not\equiv  0\,  (\text{mod } p),  \cr & \cr 
\frac{1}{4} & \quad ;  \quad  n\equiv  0\,  (\text{mod } p).
\end{array}
\right.
$$
In the above, $k_q\in \{0,...,|p|-1\}$ is the unique element with $k_qq+n \equiv 0\, (\text{mod } p )$. As we shall see, of these two definitions $s(q,p;n)$ is more closely related to the Heegaard Floer correction terms of lens spaces (compare the reciprocity formula \eqref{ReciprocityForCorrectionTerms} for $d(L(p,q),n)$ to the reciprocity formulas \eqref{KnuthReciprocity} and \eqref{ReciprocityForS} for $\sigma(q,p;n)$ and $s(q,p;n)$ respectively). 
\end{remark}
With these preliminaries out of the way, we are ready to state our main result. We state it only for lens spaces $L(p,q)$ with $p,q>0$. This hypothesis is for convenience only and does not pose any restrictions since 
\begin{equation} \label{RelationsForLensSpaces}
L(-p,q)\cong L(p,-q)\cong -L(p,q) \quad \text{ and } \quad d(-Y,\s) = -d(Y,\s).
\end{equation}
\begin{theorem} \label{main}
Let $p,q$ be two relatively prime, positive integers. Then for any integer $n$ the equality  
\begin{equation} \label{MainRelation}
d(L(p,q),n) =  2s(q,p;n)+s(q,p)-\frac{1}{2p}
\end{equation}
holds.
\end{theorem}
A similar formula was found by M. Tange \cite{Tange} involving only Dedekind, but not Dedekind-Rademacher sums. Specifically, it is shown in \cite{Tange} that 
\begin{equation} \label{TangeFormula}
d(L(p,q),n) = 3s(q,p) + \frac{p-1}{2p} + 2 \sum _{k=1}^{n} \left( \left( \frac{2q'k-1}{2p}    \right)\right), 
\end{equation}
where $q'\in\{1,...,p-1\}$ is uniquely determined by $qq'\equiv 1 \, (\text{mod } p)$. In comparisson, the compactness of formula \eqref{MainRelation} owes to the novel use of Dedekind-Rademacher sums and to the use of the first Bernoulli polynomial in place of the sawtooth function. A formula similar to \eqref{TangeFormula} was found by A. N\' emethi \cite{Nemethi} for $3$-manifolds resulting from negative definite plumbings.

Before proceeding, we remark that the utility of the correction terms of lens spaces goes beyond lens spaces themselves. For example, in \cite{OzsvathSzabo21} Ozsv\'ath and Szab\'o show that certain correction terms of the $3$-manifold $S^3_{p/q}(K)$ gotten by $p/q$-framed Dehn surgery on an $L$-knot $K\subset S^3$, can be computed as $d(S^3_{p/q}(K),n) = d(L(p,q),n)-2t_{\left\lfloor \frac{|n|}{q}\right\rfloor}(K)$. Here $t_i(K)=\sum_{j=1}^\infty ja_{|i|+j}$ are the torsion coefficients of $K$, with $a_j$ being the coefficients of the Alexander polynomial $\Delta_K(t)$ of $K$, written as $\Delta_K(t) = a_0 +\sum_{j>0} a_j(t^j+t^{-j})$.  

This situation for the correction terms is not unlike that of the Casson-Walker invariant $\lambda (S^3_{p/q}(K))$, cf. \cite{Walker},  which can be computed for $S^3_{p/q}(K)$ as $\lambda (S^3_{p/q}(K)) = \lambda (L(p,q))-\frac{q}{p} \Delta ''_K(1)$. With regards to the Casson-Walker invariant, we have the following consequence of Theorem \ref{main}.
\begin{corollary} \label{RelationOfCorrectionTermsAndCassonInvariant}
The  average of the correction terms $d(L(p,q),\s)$ over all spin$^{\mathbb C}$-structures $\s \in Spin^{\mathbb C}(L(p,q))$, equals the Casson-Walker invariant $\lambda (L(p,q))$ of the lens space $L(p,q)$, i.e. 
\begin{equation} \label{RelationToCassonInvariant}
 \frac{1}{|p|}  \sum _{\s \in Spin^{\mathbb C}(L(p,q))} d(L(p,q),\s) = \lambda (L(p,q)).
 \end{equation}
\end{corollary}
Relation \eqref{RelationToCassonInvariant} was first proved by J. Rasmussen \cite{Rasmussen}. We mention it here as it follows from Theorem \ref{main} by purely number-theoretical methods.  

The next theorem is based on the reciprocity formula \eqref{ReciprocityForCorrectionTerms} alone, and can be proved without Theorem \ref{main}. However the technique employed in its proof borrows heavily from analogous proofs in the context of Dedekind and Dedekind-Rademacher sums \cite{JabukaRobinsWang} and is therefore indirectly inspired by Theorem \ref{main}.  
\begin{theorem} \label{main2} Let $p,q$ be two relatively prime integers and let $n_1,n_2 \in \mathbb Z/p\mathbb Z$. 
\begin{itemize}
\item[(a)] If $d(L(p,q),n_1) = d(L(p,q),n_2)$ then 
$$ p \, | \, 2(n_1-n_2)(n_1+n_2-q+1).$$
\item[(b)] If $d(L(p,q),n_1) = - d(L(p,q),n_2)$ then 
$$ p \, | \, \left( (n_1-n_2)^2+ (n_1+n_2-q+1)^2\right).$$

\end{itemize}
\end{theorem}
\begin{corollary} \label{CaseOfPAPrime}
Let $p,q$ be relatively prime integers and assume that $p$ is a prime. Then the function 
$$\s \mapsto d(L(p,q),\s) : \frac{Spin^{\mathbb C}(L(p,q))}{\mathbb Z/2 \mathbb Z} \to \mathbb Q$$
is injective (where the action of $\mathbb Z/2\mathbb Z$ on $Spin^{\mathbb C}(L(p,q))$ is by conjugation $\s \mapsto \bar \s$). Additionally, the only possible vanishing correction term of $L(p,q)$ with $p$ odd is associated to the unique spin-structure of $L(p,q)$.
\end{corollary}
\begin{example}
When $p$ is a prime, Corollary \ref{CaseOfPAPrime} says that either none or only one of the correction terms of $L(p,q)$ can vanish. Both of these possibilities occur. For instance considering  $L(5,1)$ and $L(5,2)$, one finds their correction terms to be  
$$\left\{ 1, \frac{1}{5},-\frac{1}{5},-\frac{1}{5},\frac{1}{5}  \right\}\quad \quad \text{ and } \quad \quad \left\{ \frac{2}{5},\frac{2}{5},-\frac{2}{5},0,-\frac{2}{5}  \right\}.$$
\end{example}
\begin{example}
The conclusion about the injectivity of the map $\s \mapsto d(L(p,q),\s):Spin^{\mathbb C}/(\mathbb Z/2\mathbb Z) \to \mathbb Q$ from Corollary \ref{CaseOfPAPrime}, fails in general if $p$ is not a prime. For instance
$$ d(L(27,7),9)=d(L(27,7),15)=d(L(27,7),18)=d(L(27,7),24)=-\frac{1}{6},$$
with $\{9,24\}$ and $\{15,18\}$ being pairs of conjugate spin$^{\mathbb C}$-structures. The image of $n\mapsto d(L(27,7),n)$ has cardinality $12$, falling short of the theoretically possible maximum of $14$.
\end{example}
An important application of the Heegaard Floer correction terms has been towards obstructing a rational homology $3$-sphere $Y$ from bounding a rational homology $4$-ball $X$. An easy algebro-topological obstruction for this happening is that the order of $H^2(Y;\mathbb Z)$ be a square, say $p^2$, and that the image of $H^2(X;\mathbb Z) \to H^2(Y;\mathbb Z)$ have order $p$. Under compatible affine identifications of $Spin^{\mathbb C}(X)$ and $Spin^{\mathbb C}(Y)$ with $H^2(X;\mathbb Z)$ and $H^2(Y;\mathbb Z)$ respectively, it was shown in \cite{OzsvathSzabo4} by Ozsv\'ath and Szab\'o that all  spin$^{\mathbb C}$-structures in Im$\left( H^2(X;\mathbb Z) \to H^2(Y;\mathbb Z)\right)$ have vanishing associated correction terms. Our next corollary to Theorem \ref{main2} shows that this number of vanishing correction terms is maximal among lens spaces with second cohomology of square order.
\begin{corollary} \label{ClaimAboutVanishingCorrectionTermsWithPASquare}
Let $p,q$ be relatively prime integers. Then $L(p^2,q)$ has at most $|p|$ vanishing correction terms. 
\end{corollary}
\begin{remark}
We close this section by pointing out that equality \eqref{MainRelation} is indeed false if one drops the positivity condition on $p,q$. Corresponding formulas for $d(L(p,q),n)$ with one or both of $p,q$ negative can easily be derived from the proofs presented in the next section, but we didn't see much merit in doing so in lieu of \eqref{RelationsForLensSpaces}. 
\end{remark}

{\bf Acknowledgements } We thank Paolo Lisca for helpful comments and email correspondence. 
\section{Proofs}
\subsection{Proof of Theorem \ref{main}} As already noted in the previous section,  both sides of the equality \eqref{MainRelation} depend on $q$ and $n$ only through their modulus with respect to $p$. This allows us to assume, without loss of generality, that $0<q<p$ and $0\le n <p$.  

In \cite{Knuth}, D. Knuth proves the following reciprocity formula for $\sigma (q,p;n)$, valid under our assumptions of $0<q<p$ and $0\le n <p$:
\begin{equation} \label{KnuthReciprocity}
\sigma (q,p;n)+\sigma (p,q;n) = \frac{n^2}{2pq} + \frac{1}{12}\left( \frac{q}{p}+\frac{1}{pq} + \frac{p}{q}\right)-\frac{1}{2} \left\lfloor \frac{n}{q} \right\rfloor - \frac{1}{4} e(q,n),
\end{equation}
where $e(q,n)$ is defined as 
$$ e(q,n) = \left\{
\begin{array}{cl}
1 & \quad ; \quad n=0 \text{ or } n \not \equiv 0 \, (\text{mod } q), \cr
0 & \quad ; \quad n>0 \text{ and } n \equiv 0 \, (\text{mod } q). 
\end{array}
\right.
$$
\begin{lemma}
Let $p,q$ be relatively prime, positive integers with $q<p$ and let $n\in \{0,...,p-1\}$ be arbitrary. Then 
\begin{equation} \label{ReciprocityForS}
s(q,p;n)+s(p,q;n) =  \frac{n^2}{2pq} + \frac{1}{12}\left( \frac{q}{p}+\frac{1}{pq} + \frac{p}{q}\right) +\frac{n}{2}\left( \frac{1}{pq} - \frac{1}{q}  -\frac{1}{p}\right)  +\frac{1}{4} .
\end{equation}
\end{lemma}
\begin{proof}
We divide the proof into three cases, depending on the choice of $n$ used. In each case we rely on the  reciprocity formula \eqref{KnuthReciprocity} for $\sigma (q,p;n)$ and on the relation between $s(q,p;n)$ and $\sigma(q,p;n)$ elucidated in Remark \ref{RemarkOnDifferentDedekindRademacherSum}. 

{\bf Case of $\mathbf{n=0}$.} When $n=0$ then $\sigma (q,p;0)+\sigma (p,q;0) = -\frac{1}{4}+\frac{1}{12}\left( \frac{q}{p}+\frac{1}{pq} + \frac{p}{q}\right) $ while $s(q,p;0)-\sigma(q,p;0) = \frac{1}{4}=s(p,q;0)-\sigma(p,q;0)$. From these the $n=0$ case of the lemma follows. 

{\bf Case of $\mathbf{n>0}$ and $\mathbf{n\equiv 0\,(\text{mod } q)}$.} In this case $e(q,n)=0$ so that 
\begin{equation} \label{aux1}
\sigma(q,p;n)+\sigma(p,q;n) = \frac{n^2}{2pq} + \frac{1}{12}\left( \frac{q}{p}+\frac{1}{pq} + \frac{p}{q}\right)- \frac{n}{2q}.
\end{equation}
On the other hand, it is easy to see that $k_q=p-\frac{n}{q}$, which is used in the first of the following two relations:
\begin{align} \label{aux2}
s (q,p;n)-\sigma(q,p;n) & = -\frac{1}{2}\left( \bar B_1 \left(\frac{n}{p}\right)+\bar B_1\left( \frac{k_q}{p}\right) \right)\cr
& = -\frac{1}{2}\left( \frac{n}{p}-\frac{1}{2}+\bar B_1\left( \frac{p-\frac{n}{q}}{p}\right) \right) \cr
& = -\frac{1}{2}\left( \frac{n}{p}-\frac{1}{2}+\bar B_1\left( -\frac{n}{pq}\right) \right) \cr
& = -\frac{1}{2}\left( \frac{n}{p}-\frac{1}{2} -\frac{n}{pq}-\left\lfloor -\frac{n}{pq}\right\rfloor -\frac{1}{2} \right) \cr
& = -\frac{1}{2}\left( \frac{n}{p}-\frac{1}{2} -\frac{n}{pq}+1 -\frac{1}{2} \right) \cr
& = \frac{n}{2pq}  -\frac{n}{2p} \cr
s(p,q;n)-\sigma(p,q;n) & = \frac{1}{4}
\end{align}
Relations \eqref{aux1} and \eqref{aux2} yield the claim of the lemma for the present special case. 

{\bf Case of $\mathbf{n>0}$ and $\mathbf{n\not\equiv 0\,(\text{mod } q)}$.} In this setting $e(q,n)=1$ giving 
\begin{align} \label{aux3}
\sigma (q,p;n)+\sigma (p,q;n) & = \frac{n^2}{2pq} + \frac{1}{12}\left( \frac{q}{p}+\frac{1}{pq} + \frac{p}{q}\right)-\frac{1}{2} \left\lfloor \frac{n}{q} \right\rfloor - \frac{1}{4}, \cr
s(q,p;n)+s(p,q;n) & = \sigma (q,p;n)+\sigma (p,q;n)  \cr
& \quad \quad \quad - {\textstyle \frac{1}{2} \left( \bar B_1 \left(\frac{n}{p}\right)+ \bar B_1 \left(\frac{n}{q}\right) + \bar B_1\left( \frac{k_q}{p}\right)+ \bar B_1\left( \frac{k_p}{q}\right) \right). }
\end{align}
Before combining these two, we work on simplyfying the right-hand side of the second equation. Firstly, relying on our assumption of $0\le n<p$ and hence on $\lfloor \frac{n}{p}\rfloor$=0, we obtain 
\begin{align} \label{aux4}
 \bar B_1 \left(\frac{n}{p}\right)+ \bar B_1 \left(\frac{n}{q}\right) & = \frac{n}{p}+\frac{n}{q}-\left\lfloor \frac{n}{q}\right\rfloor - 1.
\end{align}
Pick integers $a,b$ such that $ap+bq=1$ and arrange that $b\in\{0,...,p-1\}$. Observe that $k_q\equiv -bn\, (\text{mod } p)$ and $k_p\equiv -an\, (\text{mod } q)$ giving us
\begin{align} \label{aux5}
 \bar B_1 \left(\frac{k_q}{p}\right)+ \bar B_1 \left(\frac{k_p}{q}\right) & =  \bar B_1 \left(\frac{-bn}{p}\right)+ \bar B_1 \left(\frac{-an}{q}\right), \cr
 & = -\frac{bn}{p} -\left\lfloor \frac{-bn}{p}\right\rfloor-\frac{an}{q} -\left\lfloor \frac{-an}{q}\right\rfloor -1, \cr
 & = -\frac{n}{pq}-\left\lfloor \frac{-qbn}{pq}\right\rfloor -\left\lfloor \frac{-an}{q}\right\rfloor -1,\cr
 & = -\frac{n}{pq}-\left\lfloor \frac{apn-n}{pq}\right\rfloor -\left\lfloor \frac{-an}{q}\right\rfloor -1,\cr
 & = -\frac{n}{pq}-\left\lfloor \frac{an}{q} - \frac{n}{pq}\right\rfloor -\left\lfloor \frac{-an}{q}\right\rfloor -1,\cr
 & = -\frac{n}{pq}-\left\lfloor \frac{an}{q} \right\rfloor -\left\lfloor \frac{-an}{q}\right\rfloor -1,\cr
 & = -\frac{n}{pq}.
 \end{align}
In the above we used the assumption that $n\not\equiv 0\,(\text{mod } q)$. For if we write $\frac{an}{q} =\ell + \frac{r}{q}$ with $\ell \in \mathbb Z$ and $r\in\{1,...,q-1\}$, then  $\frac{an}{q}-\frac{n}{pq} = \ell + \frac{rp-n}{pq}$ which, along with $1\le n\le p-1$ gives $\lfloor \frac{an}{q}-\frac{n}{pq} \rfloor = \ell = \lfloor \frac{an}{q}\rfloor$. Finally, combining \eqref{aux3}, \eqref{aux4} and \eqref{aux5} gives
$$s(q,p;n)+s(p,q;n) =  \frac{n^2}{2pq} + \frac{1}{12}\left( \frac{q}{p}+\frac{1}{pq} + \frac{p}{q}\right) + \frac{n}{2pq}  -\frac{n}{2p}-\frac{n}{2q}+\frac{1}{4}, $$
completing the proof of the lemma.
\end{proof}

The classical reciprocity formula for the Dedekind sum $s(q,p)$ can be found for example in \cite{GrosswaldRademacher}, and takes the form
\begin{equation} \label{DedekindReciprocity}
s (q,p)+s(p,q) = -\frac{1}{4} +  \frac{1}{12}\left( \frac{q}{p}+\frac{1}{pq} + \frac{p}{q}\right).
\end{equation}

We now have all ingredients in place to finish the proof of Theorem \ref{main}. Define $T(q,p;n)$ as 
$$T(q,p;n) = 2s(q,p;n)+s(q,p)-\frac{1}{2p}.$$
Then the reciprocity formulas \eqref{ReciprocityForS} and \eqref{DedekindReciprocity} lead to the reciprocity formula for  $T(q,p;n)$:
$$T(q,p;n)+T(p,q;n) = \frac{1}{4} +  \frac{n^2}{pq} + {n}\left( \frac{1}{pq} - \frac{1}{q}  -\frac{1}{p}\right)   + \frac{1}{4}\left( \frac{q}{p}+\frac{1}{pq} + \frac{p}{q}\right)  -\frac{1}{2p}-\frac{1}{2q}.$$
Clearly $T(q,p;n)$ satisfies the same reciprocity formula \eqref{ReciprocityForCorrectionTerms}  as $d(L(p,q),n)$, it depends on $q$ and $n$ only through their modulus with respect to $p$, and additionally $T(1,1;0) = 2s(1,1;0)-s(1,1)-\frac{1}{2} = 2\cdot \frac{1}{4} -0 -\frac{1}{2} = 0 $ which agrees with $d(L(1,1),0) = 0$. Accordingly we obtain $d(L(p,q),n) = T(q,p;n)$ thereby completing the proof of Theorem \ref{main}. 
\subsection{Proof of Corollary \ref{RelationOfCorrectionTermsAndCassonInvariant}} 
This is direct computation, using as input Theorem \ref{main} and the following formula, valid for all real numbers $x$ and for any positive integer $p$, see \cite{Knuth}:
\begin{equation} \label{FormulaForSums}
\sum _{n=0}^{p-1} \left(\left(  \frac{x+n}{p}\right)\right) = ((x)).
\end{equation}
Using \eqref{FormulaForSums} we can evaluate the sums: 
\begin{align} \nonumber
\sum_{n=0}^{p-1} \bar B_1\left(\frac{kq+n}{p} \right) & = -\frac{1}{2}+\sum _{n=0}^{p-1}  \left(\left(  \frac{kq+n}{p}\right)\right) = -\frac{1}{2} +((kq)) = -\frac{1}{2}, \cr 
\sum _{k=0}^{p-1} \bar B_1\left(\frac{k}{p} \right) & =-\frac{1}{2}+\sum _{k=0}^{p-1}  \left(\left(  \frac{k}{p}\right)\right) = -\frac{1}{2}. 
\end{align}
From these last two formulas we can then obtain a simplification for $\displaystyle \sum _{n=0}^{p-1} s(q,p;n)$:
\begin{align} \label{SumForDedekindRademacher}
\sum _{n=0}^{p-1} s(q,p;n) & = \sum _{n,k=0}^{p-1} \bar B_1\left(\frac{kq+n}{p} \right)  \bar B_1\left(\frac{k}{p} \right), \cr
& = -\frac{1}{2}  \sum _{k=0}^{p-1}  \bar B_1\left(\frac{k}{p} \right),  \cr
& = \frac{1}{4}.
\end{align}
Using \eqref{SumForDedekindRademacher} and Theorem \ref{main}, we can now complete the proof of Corollary \ref{RelationOfCorrectionTermsAndCassonInvariant}:

\begin{align} \nonumber
 \sum _{\s \in Spin^{\mathbb C}(L(p,q))}  d(L(p,q),\s) & = \sum _{n=0}^{p-1} d(L(p,q),n), \cr
 & = \sum _{n=0}^{p-1} \left( 2 s(q,p;n) +s(q,p) -\frac{1}{2p} \right), \cr
& = 2 \left(  \sum _{n=0}^{p-1} s(q,p;n) \right)+ p\cdot s(q,p) - \frac{1}{2}, \cr
& = \frac{1}{2} + p\cdot s(q,p) - \frac{1}{2}, \cr
& = p\cdot s(q,p).
\end{align}
Since according to \cite{Walker} the Casson-Walker invariant $\lambda(L(p,q))$ equals $s(q,p)$, the corollary follows.
\subsection{Proof of Theorem \ref{main2}}
We start the proof of Theorem \ref{main2} with a lemma. This lemma can be proved directly from the definition of the correction terms \cite{OzsvathSzabo4}, but we choose to use Theorem \ref{main} to provide an alternative proof. 
\begin{lemma}  \label{LemmaAboutIntegrality}
For any choice of relatively prime integers $p,q$ and for any integer $n$, the quantity $2p \cdot  d(L(p,q),n))$ is an integer. 
\end{lemma}
\begin{proof}
Without loss of generality we assume that $0<q<p$ and that $n\in \{0,...,p-1\}$. If $n=0$ then 
\begin{align} \nonumber
2p \cdot d(L(p,q),0) & = 2p\cdot \left(2s(q,p;0) + s(q,p)-\frac{1}{2p}  \right), \cr
& = 2p\cdot  \left(3s(q,p)+\frac{1}{2} -\frac{1}{2p}  \right), \cr
& = 6p\cdot s(q,p)+p-1. 
\end{align}
As shown in \cite{GrosswaldRademacher}, $6p\cdot s(q,p) \in \mathbb Z$, proving the lemma for $n=0$. 

If $n>0$ then
\begin{align} \nonumber
2p \cdot d(L(p,q),n) & \equiv 2p\cdot \left( 2s(q,p;n)+s(q,p) - \frac{1}{2p}\right) \, (\text{mod } 1), \cr
& \equiv 2p\cdot  {\textstyle  \left[\sum_{k=0}^{p-1}  2 \bar B_1\left( \frac{kq+n}{p}  \right) \bar B_1\left( \frac{k}{p}  \right) + \sum _{k=0}^{p-1} \left( \left(  \frac{kq}{p} \right)\right) \left( \left(  \frac{k}{p} \right)\right)\right] \, (\text{mod } 1),} \cr
&\equiv 4p {\textstyle  \sum _{k=0}^{p-1} \left( \frac{kq+n}{p}-\frac{1}{2}   \right)\left(\frac{k}{p} - \frac{1}{2} \right) + 2p \sum_{k=1}^{p-1} \left( \frac{kq}{p} -\frac{1}{2}  \right)  \left( \frac{k}{p} -\frac{1}{2}  \right) \, (\text{mod } 1),} \cr
& \equiv 4p {\textstyle  \sum_{k=0}^{p-1} \left( k^2 \frac{q}{p^2} + k\frac{n}{p^2} \right)    + 2p \sum _{k=1}^{p-1} \left(   k^2 \frac{q}{p^2}    +\frac{1}{4}         \right)                \, (\text{mod } 1),} \cr
& \equiv 4p {\textstyle  \left( \frac{(2p-1)(p-1)q}{6p}  + \frac{(p-1)n}{2p} \right)    + 2p \left(   \frac{(2p-1)(p-1)q}{6p}   +\frac{p-1}{4}         \right)                \, (\text{mod } 1), }\cr
& \equiv  {\textstyle (2p-1)(p-1)q+ 2(p-1)n   + \frac{p(p-1)}{2}    \, (\text{mod } 1), }\cr
& \equiv  0    \, (\text{mod } 1),
\end{align}
as claimed.
\end{proof}
To prove part (a) of the theorem, suppose now that $0<q<p$ are two relatively prime integers and that $d(L(p,q),n_1) = d(L(p,q),n_2)$ for some choice of $n_1,n_2\in \{0,...,p-1\}$. Taking the reciprocity formula \eqref{ReciprocityForCorrectionTerms} first with $n=n_1$, then with $n=n_2$, multiplying each by $2pq$ and subtracting the two equations, yields
$$
p \cdot (2q \, d(L(q,p),n_1) - 2q\, d(L(q,p),n_2))  = 2(n_1^2-n_2^2) + 2(n_1-n_2) (1-q-p).
$$
Reducing the previous equation modulo $p$, taking Lemma \ref{LemmaAboutIntegrality} into account, gives 
$$ 0 \equiv  2(n_1-n_2)(n_1+n_2+1-q) \, (\text{mod } p),$$
which is the claim of part (a). 

For part (b) we use the same approach but rather than subtracting, we add the two versions of the reciprocity law \eqref{ReciprocityForCorrectionTerms}, and obtain
\begin{align} \nonumber
p \cdot (2q \, d(L(q,p),& n_1) +  2q\, d(L(q,p),n_2)) = \cr
 = & pq+ 2(n_1^2+n_2^2) +  2(n_1+n_2) (1-q-p)+(p^2 +1 +q^2) -2p-2q. \cr
\end{align}
Reducing again both sides modulo $p$ gives
\begin{align}\nonumber
0 & \equiv \left[ 2(n_1^2+n_2^2) + 2(n_1+n_2) (1-q)+(1 -q)^2\right]  (\text{mod } p) ,\cr
& \equiv \left[ (n_1^2-2n_1n_2+n_2^2) + (n_1^2+2n_1n_2+n_2^2)+2(n_1+n_2) (1-q)+(1 -q)^2 \right]  (\text{mod } p) ,\cr
& \equiv \left[ (n_1-n_2)^2 + (n_1+n_2+1-q)^2\right]  (\text{mod } p).
\end{align}
This completes the proof of Theorem \ref{main2}.
\subsection{Proof of Corollary \ref{CaseOfPAPrime}}
The corollary is trivially true for $p=2$ since the two correction terms of the lens space $L(2,1)$ are $\{\frac{1}{4}, -\frac{1}{4}\}$. Thus, we may assume that $0<q<p$ are relatively prime and that $p$ is an odd prime. Additionally suppose that $n_1,n_2\in\{0,...,p-1\}$ have the property that $d(L(p,q),n_1) = d(L(p,q),n_2)$. From part (a) of Theorem \ref{main2} we infer that $p$ must divide $2(n_1-n_2)(n_1+n_2-q+1)$ and hence that either $p$ divides $n_1-n_2$ or that $p$ divides  $n_1+n_2-q+1$. In the former case we are led to $n_1=n_2$, proving the corollary, while in the latter we get $p|(n_1+n_2-q+1)$. Since $0\le n_1,n_2<p$ and $0<q<p$, we see that the only way that $p$ can divide $n_1+n_2-q+1$ is that either $n_1+n_2-q+1=0$ or $n_1+n_2-q+1=p$ (but not both). In particular, for each $n_1$ there can be at most one $n_2$ with $d(L(p,q),n_1) = d(L(p,q),n_2)$. Since the correction terms corresponding two conjugate spin$^{\mathbb C}$-structures are equal, we infer that any such pair $(n_1,n_2)$ must correspond to a pair of conjugate spin$^{\mathbb C}$-structures. From this the asserted injectivity property follows.

If $d(L(p,q),n_1) = d(L(p,q),n_2)=0$ then by relying on part (b) of Theorem \ref{main2}, we conclude that $p$ has to divide $n_1-n_2$. For if we had instead that $p|(n_1+n_2-q+1)$ then from $p$ dividing $(n_1-n_2)^2+(n_1+n_2-q+1)^2$ we could conclude that $p$ divides $(n_1-n_2)^2$ and hence $n_1-n_2$. From this we immediately find that $n_1=n_2$. Since the unique spin-structure of $L(p,q)$ is the only self-conjugate spin$^{\mathbb C}$-structure, the claim of the corollary follows.
\subsection{Proof of Corollary \ref{ClaimAboutVanishingCorrectionTermsWithPASquare}}
Let $p,q$ be relatively prime and positive, and without loss of generality assume that $q<p^2$. If $n_1,n_2\in\{0,...,p^2-1\}$ are such that $d(L(p^2,q),n_1) = d(L(p^2,q),n_2) =0$, then parts (a) and (b) of Theorem \ref{main2} imply that $p^2$ must be diving both $2(n_1-n_2)(n_1+n_2-q+1)$ and $(n_1-n_2)^2+(n_1+n_2-q+1)^2$. Combining the two we see that $p^2$ must then also divide $((n_1-n_2) \pm (n_1+n_2-q+1))^2$ and hence that $p$ must divide $(n_1-n_2) \pm (n_1+n_2-q+1)$. Adding and subtracting these two congruences we obtain 
$$p| (2n_1 -q+1) \quad \text{ and } \quad p|(2n_2-q+1).$$
Adding and subtracting these latter congruences yet one more time, we arrive at 
\begin{equation} \label{usefulecongruence}
p| 2(n_1 -n_2) \quad \text{ and } \quad p|2(n_1+n_2-q+1).
\end{equation}
If $p$ is odd then $p|(n_1-n_2)$ showing that there can be at most $p$ solutions $n$ to $d(L(p^2,q),n)=0$. If $p$ is even, then \eqref{usefulecongruence} shows that $\frac{p}{2}$ divides both 
$n_1-n_2$ and $n_1+n_2-q+1$. We'd like to again conclude that $p|(n_1-n_2)$ as in the case of odd $p$. If $p|(n_1+n_2-q+1)$ then $p^2|(n_1+n_2-q+1)^2$ and so part (b) of Theorem \ref{main2} says that $p^2|(n_1-n_2)^2$ and hence that $p|(n_1-n_2)$, as desired. If $p$ does not divide either of $n_1-n_2$, $n_1+n_2-q+1$ then we can write 
$$n_1-n_2 = \ell _1 \cdot \frac{p}{2} \quad \text{ and } \quad n_1+n_2-q+1 = \ell _2\cdot \frac{p}{2} $$
with $\ell_1, \ell _2$ odd integers. From this, and from part (b) of Theorem \ref{main2} we get the equality 
$$2(n_1-n_2)(n_1+n_2-q+1) = 2\left( \ell_1 \cdot \frac{p}{2}\right) \cdot \left( \ell_2 \cdot \frac{p}{2}\right) = p^2\cdot \ell  $$
for some $\ell\in \mathbb Z$.  Cancelling $p^2$ on both sides of this last equation, yields the contradiction $\frac{\ell_1\ell_2}{2} = \ell$. Thus we are again led to conclude that $p$ divides $n_1-n_2$ and so the conclusion of Corollary \ref{ClaimAboutVanishingCorrectionTermsWithPASquare} follows. 


\end{document}